\documentclass[12pt,cd]{amsart}
\usepackage{latexsym}
\usepackage{amsmath,amscd,amssymb}
\usepackage{graphics}
 2

\usepackage{hyperref}
\newtheorem{theorem}{Theorem}[section]
\newtheorem{proposition}{Proposition}[section]

\newtheorem{definition}[theorem]{Definition}
\newtheorem{example}[theorem]{Example}

\newtheorem{remark}[theorem]{Remark}
\numberwithin{equation}{section}

\title{ Born-Infeld solitons, Maximal surfaces and Ramanujan's Identities}

\author{Rukmini Dey}
\address{International Centre for Theoretical Sciences, Bengaluru- 560 089, India}
\email{rukmini@icts.res.in}
\author{Rahul Kumar Singh}
\address{Harish-Chandra Research Institute, HBNI, Allahabad-211 019, India}
\email{rhlsngh498@gmail.com}
\subjclass[2010]{53A35, 53B30, 53B50}
\keywords{Born-Infeld equation, conjugate maximal surfaces, Ramanujan Identity, soliton, Weierstrass-Enneper representation }
\begin{document}

\maketitle

\begin{abstract}
We show that a Born-Infeld soliton can be realised either as a spacelike minimal graph or timelike minimal graph over a timelike plane or a combination of both away from singular points. We also obtain some exact solutions of the Born-Infeld equation from already known solutions to the maximal surface equation. Further we present a method to construct a one-parameter family of complex solitons from a given one parameter family of maximal surfaces. Finally, using Ramanujan's Identities and the Weierstrass-Enneper representation of maximal surfaces, we derive further non-trivial identities. 
\end{abstract}

\section{Introduction}
This paper explores the beautiful relationship between Born-Infeld solitons and maximal surfaces in ${\mathbb L}^3$ (Lorentz-Minkowski space) and discusses some nontrivial identities which arises as a consequence of  certain  Ramanujan identities and the Weierstrass-Enneper representation for maximal surfaces.

Any smooth function  $ \varphi(x,t) $  which is a solution to Born-Infeld equation(see \cite{whitham})
\begin{equation}\label{borninfeldequn}
(1+\varphi_{x}^{2})\varphi_{tt}-2\varphi_{x}\varphi_{t}\varphi_{xt}+(\varphi_{t}^{2}-1)\varphi_{xx}=0.
\end{equation}
is known as a Born-Infeld soliton.

A graph  $(x,t,f(x,t))$ in Lorentz-Minkowski space $ \mathbb{L}^3 :=(\mathbb{R}^3, dx^2+dt^2-dz^2)$ is maximal if it satisfies
\begin{equation}\label{maximalsurfaceequn}
(1-f_{x}^2)f_{tt}+2f_{x}f_{t}f_{xt}+
(1-f_{t}^2)f_{xx}=0,
\end{equation}
for some  smooth function $ f(x,t) $  satisfying $ f_x^2+f_t^2<1,$ see \cite{kobasingu}.
This equation is known as maximal surface equation.

It has been known that the Born-Infeld equation is related to the minimal surface equation in $ \mathbb{R}^3 $ via a wick rotation in the variable $ t $ i.e., if we replace $ t $ by $ it $  in \eqref{borninfeldequn}, we get back the minimal surface equation and vice-versa \cite{rukmini}.
This fact has been used by the authors in \cite{kristina} to obtain some exact solutions of the Born-Infeld equation.

In this paper  we observe that the Born-Infeld equation is also related to the maximal surface equation by a wick rotation in the variable $ x $ i.e., if we replace $x$ by $ ix $ and define $ f(x,t):=\varphi(ix,t),$ in \eqref{borninfeldequn}, we get back the maximal surface equation \eqref{maximalsurfaceequn} and vice-versa \cite{rahul}.

Using this interrelation the authors of this paper had earlier rederived the Weierstrass-Enneper representation for minimal surfaces and maximal surfaces, assuming that the Gauss map for such surfaces is one-to-one(see \cite{rukmini}, \cite{rahul}).\\
Also, the first author of this paper and collaborator had obtained a one parameter family for Born-Infeld solitons from a given one parameter family of minimal surfaces \cite{rukminione}. Recently, in \cite{rukminirama}, the first author of this paper had obtained some nontrivial identities using Ramanujan Identities and Weierstrass-Enneper representation for minimal surfaces.

In this paper, we further explore the interrelation between Born-Infeld equation and maximal surface equation and obtain some analogous results.
\begin{remark}
Maximal surface equation can be obtained from the minimal surface equation by wick rotation in both the variables and vice-versa, but in general we get complex surfaces this way. The identities use Weierstrass-Enneper representation of real maximal surfaces and hence they cannot be obtained from Weierstrass-Enneper representation of real minimal surfaces.
 \end{remark}
 
\section{Born-Infeld Solitons}

Consider the Lorentz-Minkowski space $ \mathbb{L}^3,$ assuming that the cartesian coordinates are $ (x,y,z) $, then the Lorentzian metric is denoted by $ dx^2+dy^2-dz^2 $ or $ \langle,\rangle_{L} $. Then a graph in $ \mathbb{L}^3 $ over a domain of the timelike plane $ \{x=0\} $ has the form 
\begin{equation}\label{graphovertime}
 X(y,z)= (\varphi(y,z),y,z), 
\end{equation}  where $ \varphi: \Omega \subset \mathbb{R}^2\rightarrow \mathbb{R} $ is a smooth function \cite{magid}. 
A graph in $ \mathbb{L}^3 $ is said to be minimal if its mean curvature vanishes everywhere (i.e. $ H\equiv 0$). 
For the definitions of the normal vector $ N $ and the mean curvature $ H $ for a non-degenerate surface in Lorentz-Minkowski space, (see page no.$ 34 $ of \cite{lopez}).
\begin{proposition}
The solutions of \eqref{borninfeldequn}, i.e., Born-Infeld solitons can be represented as a spacelike minimal graph or timelike minimal graph over a domain in timelike plane or a combination of both away from singular points (points where tangent plane degenerates),  i.e., points where the determinant of the coefficients of first fundamental form vanishes. 
\end{proposition}

\begin{proof}

Coefficients of first fundamental form for \eqref{graphovertime} are
\begin{center}
$ E=\varphi_{y}^{2}+1 $
~~,~~$ G=\varphi_{z}^{2}-1 $
~~,~~$ F=\varphi_{y}\varphi_{z} $
\end{center}
and determinant of the coefficients of the first fundamental form is $EG-F^2=-\varphi_y^2+\varphi_z^2-1.$ In general we can have $-\varphi_y^2+\varphi_z^2-1=0$ (tangent plane degenerates). But when $-\varphi_y^2+\varphi_z^2-1\neq 0,$ one can define the normal vector $ N $ and it is given by
\begin{center}
$ N=\left(\frac{1}{\sqrt{\vert 1+\varphi_{y}^{2}-\varphi_{z}^{2} \vert}},\frac{-\varphi_{y}}{\sqrt{\vert 1+\varphi_{y}^{2}-\varphi_{z}^{2} \vert}},\frac{\varphi_{z}}{\sqrt{\vert 1+\varphi_{y}^{2}-\varphi_{z}^{2} \vert}}\right).$
\end{center}
 
Therefore
\begin{center}
$ \langle N,N\rangle_{L}=\frac{1+\varphi_{y}^{2}-\varphi_{z}^{2}}{\vert 1+\varphi_{y}^{2}-\varphi_{z}^{2} \vert}. $
\end{center}

If $ 1+\varphi_{y}^{2}-\varphi_{z}^{2}>0  $, we have $ \langle N,N\rangle_{L}=1,$ then the graph is timelike. On the other hand if $ \langle N,N\rangle_{L}=-1 $, i.e. $ 1+\varphi_{y}^{2}-\varphi_{z}^{2}<0,  $ then the graph is spacelike.\\
Now we can easily compute coefficients of second fundamental form, they are given by
\begin{center}
$ e=\frac{\varphi_{yy}}{\sqrt{\vert 1+\varphi_{y}^{2}-\varphi_{z}^{2} \vert}} $
~~,~~
$ g=\frac{\varphi_{zz}}{\sqrt{\vert 1+\varphi_{y}^{2}-\varphi_{z}^{2} \vert}} $
~~,~~
$f=\frac{\varphi_{yz}}{\sqrt{\vert 1+\varphi_{y}^{2}-\varphi_{z}^{2} \vert}}, $
\end{center}
~\\~
here we see $ EG-F^2=-1-\varphi_{y}^{2}+\varphi_{z}^{2} $, and if $ EG-F^2>0, $ i.e. $ 1+\varphi_{y}^{2}-\varphi_{z}^{2}< 0 $ the graph is spacelike and if $ EG-F^2<0 $, i.e., $ 1+\varphi_{y}^{2}-\varphi_{z}^{2}> 0 $, then the graph is timelike. In any case we know that the mean curvature for a surface in $ \mathbb{L}^3 $ is given by( see page no. $ 40 $ of \cite{lopez}).
\begin{center}
$ H=\dfrac{\epsilon}{2}\left(\dfrac{eG-2fF+gE}{EG-F^2}\right), $
\end{center}
where $ \epsilon=1 $  if the surface is timelike, $ \epsilon=-1 $ if the surface is spacelike.
So for the spacelike graph over a timelike plane, we have
\begin{equation*}
 H=-\dfrac{1}{2}\dfrac{(1+\varphi_{y}^{2})\varphi_{zz}-2\varphi_{y}\varphi_{z}\varphi_{yz}+(\varphi_{z}^{2}-1)\varphi_{yy}}{(-1-\varphi_{y}^{2}+\varphi_{z}^{2})^{\frac{3}{2}}}, 
\end{equation*}
and  for the timelike graph over timelike plane, we have
\begin{equation*}
 H=-\dfrac{1}{2}\dfrac{(1+\varphi_{y}^{2})\varphi_{zz}-2\varphi_{y}\varphi_{z}\varphi_{yz}+(\varphi_{z}^{2}-1)\varphi_{yy}}{(1+\varphi_{y}^{2}-\varphi_{z}^{2})^{\frac{3}{2}}}. 
\end{equation*} 
So if the mean curvature $ H $ for the spacelike graph or timelike graph over a timelike plane is zero, we get 
\begin{equation*}
(1+\varphi_{y}^{2})\varphi_{zz}-2\varphi_{y}\varphi_{z}\varphi_{yz}+(\varphi_{z}^{2}-1)\varphi_{yy}=0.
\end{equation*}
By renaming the variables $ y,z $ as $ x,t $, we get

$$(1+\varphi_{x}^{2})\varphi_{tt}-2\varphi_{x}\varphi_{t}\varphi_{xt}+(\varphi_{t}^{2}-1)\varphi_{xx}=0.$$
This is nothing but the Born-Infeld equation.
\end{proof}

Now we will give an example of a Born-Infeld soliton which has some points where the determinant of the coefficients of  first fundamental form vanishes, i.e. it has the points where the tangent plane is lightlike (tangent plane degenerates).
\begin{example}
Consider the graph $ X(y,z)=(x=\sinh^{-1}(\sqrt{z^2-y^2}),y,z) $. Then we can easily check that it satisfies the Born-Infeld equation. Also, its tangent planes  degenerates precisely at the points $ (x,y,z)\in \mathbb{L}^3 $ where $ x=0 $ and $ y=\pm z. $ \\
This Born-Infeld soliton can be obtained from the elliptic catenoid (a maximal surface, see \cite{mira}), by wick rotation (a concept which we describe in the next section) and renaming the variables.
\end{example}

\section{Wick rotation of maximal surface equation}
In this section we are going to obtain some solutions to the Born-Infeld equation \eqref{borninfeldequn} from some of the already known solutions to the maximal surface equation \eqref{maximalsurfaceequn}.
 Suppose if $ f(x,t), $ is a solution to the maximal surface equation \eqref{maximalsurfaceequn}, then we obtain a solution to Born-Infeld equation \eqref{borninfeldequn}, by defining $ \varphi(x,t):=f(ix,t). $
  Some of the solutions will be real-valued and some of them will be complex.
 \vskip 2mm
 
 \textit{Wick Helicoid of the first kind:} Consider helicoid of the first kind (see\cite{mira})
$$ f(x,t)=\frac{1}{k}\tan^{-1}(\frac{t}{x}), k\neq 0~~~\text{and}~~~ k\in \mathbb{R}. $$ Then
\begin{center}
$ \varphi(x,t):=f(ix,t)=-\frac{i}{k}\tanh^{-1}(\frac{t}{x}) $
\end{center}
a complex-valued solution to the Born-Infeld equation.

\vskip 2mm

\textit{Wick Helicoid of the second kind:} Next consider helicoid of the second kind (see\cite{mira}) $$ f(x,t)=x\tanh kt, k\neq 0 ~~~\text{and}~~~ k\in \mathbb{R} $$
\begin{center}
$  \varphi(x,t):=f(ix,t)=ix\tanh kt$.
\end{center}
which is again a complex valued solution to the Born-Infeld equation.

\vskip 2mm

\textit{Wick Scherk's surface of the first kind:} Consider (see \cite{kobamax}) $$f(x,t)=\ln\left(\dfrac{\cosh t}{\cosh x}\right) $$
\begin{center}
$  \varphi(x,t):=f(ix,t)=\ln\left(\dfrac{\cosh t}{\cos x}\right). $
\end{center}
Since $ \cosh t $ is always positive, this solution is conditionally real-valued, depending on the sign of $ \cos x $.

\section{One parameter family of complex solitons}

\begin{definition}
 $$~~\text{Let}~~ X_1(\tau,\bar{\tau})=(x_1(\tau,\bar{\tau}),t_1(\tau,\bar{\tau}),f_1(\tau,\bar{\tau}))~~~\text{and}~~~ X_2(\tau,\bar{\tau})=(x_2(\tau,\bar{\tau}),t_2(\tau,\bar{\tau}),f_2(\tau,\bar{\tau})) $$ be  isothermal paramerizations of two maximal surfaces, where $ X_j(\tau,\bar{\tau}):\Omega\subseteq \mathbb{C}\rightarrow \mathbb{L}^3, \tau=\tilde{u}+i\tilde{v} \in \Omega ~~\text{;}~~ j=1,2 $  such that $$ X:=X_1+iX_2:\Omega\subseteq \mathbb{C}\rightarrow \mathbb{C}^3$$ is a holomorphic mapping. Then we say that $ X_1 $ and $ X_2 $ are conjugate maximal surfaces.
\end{definition} 
It should be remarked that if the Gauss map of a given maximal surface in $\mathbb{L}^3$ is one-one, then its conjugate maximal surface exist.
 If $ X_1(\tau,\bar{\tau})=(x_1(\tau,\bar{\tau}),t_1(\tau,\bar{\tau}),f_1(\tau,\bar{\tau})) $ is a maximal surface and $ X_2(\tau,\bar{\tau})=(x_2(\tau,\bar{\tau}),t_2(\tau,\bar{\tau}),f_2(\tau,\bar{\tau})) $ its conjugate maximal surface, where $ \tau=\tilde{u}+i\tilde{v} $ is an isothermal coordinate system. Then it can be easliy shown that  
$$ X_{\theta}(\tau,\bar{\tau}):= X_1(\tau,\bar{\tau})\cos{\theta}+ X_2(\tau,\bar{\tau})\sin{\theta} $$
also defines a maximal surface for each $ \theta $. 
\begin{remark}
$$ X_{\theta}(\tau,\bar{\tau}):= X_1(\tau,\bar{\tau})\cos{\theta}+ X_2(\tau,\bar{\tau})\sin{\theta}=Re\{e^{-i\theta}X(\tau,\bar{\tau})\}  $$  corresponds to the fact that the Weiestrass-Enneper data for the maximal surface $ X_{\theta} $ is given by $e^{-i \theta}M$, where $M$ is the Weierstrass-Enneper data for $ X_1. $
\end{remark}

As we have seen earlier, if $ (x,t,f(x,t)) $ is a solution to maximal surface equation \eqref{maximalsurfaceequn}, then $ (ix,t,\varphi(x,t):=f(ix,t)) $ is a solution for Born-Infeld equation \eqref{borninfeldequn}.\\
Next, if $ X_1=(x_1,t_1,f_1) $ and $ X_2=(x_2,t_2,f_2) $ are conjugate maximal surfaces, then we define $ X_1^s=(ix_1,t_1,\varphi_1)~~\text{,}~~ X_2^s=(ix_2,t_2,\varphi_2) $ as conjugate Born-Infeld Solitons. \\

Now we digress a little. According to a known result \cite{rahul}  if  $$ X_j(\tau,\bar{\tau})=(x_j(\tau,\bar{\tau}),t_j(\tau,\bar{\tau}),f_j(\tau,\bar{\tau}))$$ for $j=1,2 $ be two   maximal surfaces, then  
     $$ x_j-it_j=F_j(\tau)+\int{\bar{\tau}}^2\overline{F_j'(\tau)}d\bar{\tau},~~~\text{ }~~~
      x_j+it_j=\overline{F_j(\tau)}+\int{\tau}^2{F_j'(\tau)}d{\tau},$$
     $$ f_j=\int{\tau}{F_j'(\tau)}d{\tau}+\int\bar{\tau}(\overline{F_j(\tau)})'d{\bar{\tau}}.$$
where $F_j$ are functions which can be derived from the Weierstrass-Enneper data.     
 
 Then 
     $$ ix_j+t_j=iF_j(\tau)+i\int{\bar{\tau}}^2\overline{F_j'(\tau)}d\bar{\tau},
     ~~~\text{ }~~~
      ix_j-t_j=i\overline{F_j(\tau)}+i\int{\tau}^2{F_j'(\tau)}d{\tau},$$
     $$ f_j=\int{\tau}{F_j'(\tau)}d{\tau}+\int\bar{\tau}(\overline{F_j(\tau)})'d{\bar{\tau}}.$$

     Now we make an isothermal change of coordinates i.e. replacing $ \tau $ by $ i\zeta $ and $ \bar{\tau}$ by $-i\bar{\zeta} $. Then
    
 \begin{align}\label{addx1t1} ix_j+t_j=iF_j(i\zeta)-\int{\bar{\zeta}}^2d(i\overline{F_j(i\zeta)})=H_j(\zeta)-\int{\bar{\zeta}}^2G_j'(\bar{\zeta})d{\bar{\zeta}},
 \end{align}
 
 \begin{align}\label{minusx1t1}  
 ix_j-t_j=i\overline{F_j(i\zeta)}-\int{\zeta}^2d(i{F_j(i\zeta)})=G_j(\bar{\zeta})-\int{\zeta}^2H_j'({\zeta})d{{\zeta}},
\end{align}

\begin{align}\label{phi1}
 f_j=\int{\zeta}d(iF_j(i\zeta))+\int\bar{\zeta}d(-i\overline{F_j(i\zeta)})=\int{\zeta}{H_j'(\zeta)}d{\zeta}+\int\bar{\zeta}(-G_j'(\bar{\zeta}))d{\bar{\zeta}},
 \end{align}

where $ H_j(\zeta)= iF_j(i\zeta)$ and $ G_j(\bar{\zeta})= i\overline{F_j(i\zeta)}$ and they satisfy $ \overline{G_j(\bar{\zeta})}=-H_j(\zeta) .$ \\

To come back to solitons, define $$ X_{\theta}^s(\zeta,\bar{\zeta})= X_1^s(\zeta,\bar{\zeta})\cos{\theta}+ X_2^s(\zeta,\bar{\zeta})\sin{\theta}, $$
then $$ X_{\theta}^s=(ix_1,t_1,\varphi_1)\cos{\theta}+(ix_2,t_2,\varphi_2)\sin{\theta},$$  we let
$$ X_{\theta}^s=(i(x_1\cos{\theta}+x_2\sin{\theta}),(t_1\cos{\theta}+t_2\sin{\theta}),(\varphi_1\cos{\theta}+\varphi_2\sin{\theta}))=(x_{\theta}^s,t_{\theta}^s,
\varphi_{\theta}^s). $$\\
Now we prove the following proposition:
 \begin{proposition} Let $X_1=(x_1,t_1,f_1)  $ and $ X_2=(x_2,t_2,f_2) $ be two conjugate maximal surfaces and let $ X_{\theta}=(x_1\cos{\theta}+x_2\sin{\theta}, t_1\cos{\theta}+t_2\sin{\theta}, f_1\cos{\theta}+f_2\sin{\theta}))=(x_{\theta},t_{\theta},
f_{\theta})$ denotes the one parameter family of maximal surfaces corresponding to $ X_1 $ and $ X_2 $. Then  $ X_{\theta}^s=(i(x_1\cos{\theta}+x_2\sin{\theta}),(t_1\cos{\theta}+t_2\sin{\theta}),(\varphi_1\cos{\theta}+\varphi_2\sin{\theta}))=(x_{\theta}^s,t_{\theta}^s,
\varphi_{\theta}^s),$ where $ \varphi_j(x_j,t_j):=f_j(ix_j,t_j) $, $ j=1,2 $ will give us a one parameter family of complex solitons i.e. for each $ \theta $ we will have a complex solution to the Born-Infeld equation \eqref{borninfeldequn}. 
 \end{proposition}
 \begin{proof}
 To show this, we show that $ X_{\theta}^s=(x_{\theta}^s,t_{\theta}^s,
\varphi_{\theta}^s)$ will give us the general solution of the Born-Infeld equation, as described in ~\cite{whitham}:
 Consider
\begin{align}
x_{\theta}^s-t_{\theta}^s=&(ix_1-t_1)\cos{\theta}+(ix_2-t_2)\sin{\theta} \nonumber \\
=&(G_1(\bar{\zeta})\cos{\theta}+G_2(\bar{\zeta})\sin{\theta})-\int({\zeta}^2
{H_1'(\zeta)}\cos{\theta}+ {\zeta}^2
{H_2'(\zeta)}\sin{\theta})d{\zeta}, \nonumber
\end{align}
where last line is obtained using \eqref{minusx1t1}.\\
If we define $ G_{\theta}^s(\bar{\zeta}):=G_1(\bar{\zeta})\cos{\theta}+G_2(\bar{\zeta})\sin{\theta} $ and $H_{\theta}^s(\zeta):= {H_1(\zeta)}\cos{\theta}+ {H_2(\zeta)}\sin{\theta} $, then $\overline{G_{\theta}^s(\bar{\zeta})}=-H_{\theta}^s(\zeta)$. Therefore
\begin{align}\label{xsminusts}
x_{\theta}^s-t_{\theta}^s
=G_{\theta}^s(\bar{\zeta})-\int{\zeta}^2{H_{\theta}^s}'({\zeta})d{\zeta},
\end{align}
in a similar manner, we can show
\begin{align}\label{xsplusts}
x_{\theta}^s+t_{\theta}^s=
H_{\theta}^s(\zeta)-\int{\bar{\zeta}}^2{G_{\theta}^s}'(\bar{\zeta})d{\bar{\zeta}}
\end{align}
and
\begin{align}\label{phis}
\varphi_{\theta}^s=\int{\zeta}{H_{\theta}^s}'(\zeta)d{\zeta}+\int\bar{\zeta}(-{G_{\theta}^s}'(\bar{\zeta}))d{\bar{\zeta}}
\end{align}
Now the expressions \eqref{xsminusts}, \eqref{xsplusts} and \eqref{phis} describes the general solution for Born-Infeld equation, see \cite{whitham}, where $ G_{\theta}^s(\bar{\zeta}) $ and $ H_{\theta}^s(\zeta) $ are such that they satisfy $\overline{G_{\theta}^s(\bar{\zeta})}=-H_{\theta}^s(\zeta)$.
\end{proof}

\section{Example}
Consider the Lorentzian helicoid
\begin{align*}
f_1(x_1,t_1)=\frac{\pi}{2}+\tan^{-1}\left(\frac{t_1}{x_1}\right).
\end{align*} 
 which is a maximal surface in the Lorentz-Minkowski space whose Gauss map is one-one. Then the W-E representation in terms of the coordinates $(\tau, \overline{\tau})$ is given by $X_1(\tau,\bar{\tau})=(x_1(\tau,\bar{\tau}),t_1(\tau,\bar{\tau}),f_1(\tau,\bar{\tau})) $ where (for details see \cite{rahul})
$$ x_1(\tau,\bar{\tau})=\frac{1}{2} Im \left(\tau-\frac{1}{\tau}\right),~~\text{ }~~t_1(\tau,\bar{\tau})= -\frac{1}{2} Re \left(\tau+\frac{1}{\tau}\right),$$
$$f_1(\tau,\bar{\tau})=-\frac{i}{2}\ln\left(\frac{\tau}{\bar{\tau}}\right).$$
and similarly for Lorentzian catenoid 
 \begin{align*}
f_2(x_2,t_2)=\sinh^{-1}(\sqrt{x_2^2+t_2^2}), 
\end{align*} 
we have (for details see\cite{rahul})
 , $ X_2(\tau,\bar{\tau})=(x_2(\tau,\bar{\tau}),t_2(\tau,\bar{\tau}),f_2(\tau,\bar{\tau}))$, where 
$$ x_2(\tau,\bar{\tau})=-\frac{1}{2} Re \left(\tau-\frac{1}{\tau}\right),~~\text{ }~~
t_2(\tau,\bar{\tau})= -\frac{1}{2} Im\left(\tau+\frac{1}{\tau}\right),$$
$$f_2(\tau,\bar{\tau})=-\frac{1}{2}\ln(\tau\bar{\tau}).$$

Then
$$ x_1+ix_2=\frac{-i}{2}\left(\tau-\frac{1}{\tau}\right),~~\text{ }~~t_1+it_2=\frac{-1}{2}\left(\tau+\frac{1}{\tau}\right),~~\text{ }~~f_1+if_2=-i\ln{\tau}.$$
Thus we see that $X_1+iX_2:=(\frac{-i}{2}\left(\tau-\frac{1}{\tau}\right),\frac{-1}{2}\left(\tau+\frac{1}{\tau}\right),-i\ln{\tau})$ is a holomorphic mapping on a common domain of $\mathbb{C}-\{0\}  $. Therefore, the Lorentzian helicoid and Lorentzian catenoid are conjugate maximal surfaces. Now $$ X_{\theta}(\tau,\bar{\tau}):= X_1(\tau,\bar{\tau})\cos{\theta}+ X_2(\tau,\bar{\tau})\sin{\theta} $$ gives a one parameter family of maximal surfaces. We have  
$$ix_1-t_1=\frac{1}{2}\left(\frac{1}{\bar{\tau}}+\tau \right),~~\text{and }~~ 
ix_1+t_1=-\frac{1}{2}\left(\frac{1}{\tau}+\bar{\tau} \right).$$
and
$$ix_2-t_2=\frac{i}{2}\left(\frac{1}{\bar{\tau}}-\tau \right),~~\text{and }~~  
ix_2+t_2=\frac{i}{2}\left(\frac{1}{\tau}-\bar{\tau} \right).$$

If we replace $ \tau $ by $ i\zeta $ and $ \bar{\tau} $ by $ -i\bar{\zeta} $ we get
\begin{align}
ix_1-t_1=\frac{i}{2}\left(\frac{1}{\bar{\zeta}}+\zeta \right)~~~\text{;}~~~ 
ix_1+t_1=\frac{i}{2}\left(\frac{1}{\zeta}+\bar{\zeta} \right)
\end{align}
and
\begin{align}
f_1(\zeta,\bar{\zeta})=-\frac{i}{2}\ln\left(\frac{\zeta}{\bar{\zeta}}\right).
\end{align}
\begin{align}
ix_2-t_2=-\frac{1}{2}\left(\frac{1}{\bar{\zeta}}-\zeta \right)~~~\text{;}~~~ 
ix_2+t_2=\frac{1}{2}\left(\frac{1}{\zeta}-\bar{\zeta} \right)
   \end{align}
\begin{align}
f_2(\zeta,\bar{\zeta})=-\frac{1}{2}\ln(\zeta\bar{\zeta}).
\end{align}

Now we are going to compute the functions $ G_{\theta}^s(\bar{\zeta}) $ and $ H_{\theta}^s(\zeta) $ which will give our required one parameter family of complex solitons corresponding to the one parameter family of maximal surfaces mentioned above. We first compute
\begin{align}\label{xthetaminus}
x_{\theta}^s-t_{\theta}^s=&(ix_1-t_1)\cos{\theta}+(ix_2-t_2)\sin{\theta} \nonumber \\
=&\frac{i}{2\bar{\zeta}}e^{i\theta}+\frac{i\zeta}{2}e^{-i\theta},
\end{align}
next we compute 
\begin{align}\label{xthetaplus}
x_{\theta}^s+t_{\theta}^s=&(ix_1+t_1)\cos{\theta}+(ix_2+t_2)\sin{\theta} \nonumber \\
=&\frac{i}{2\zeta}e^{-i\theta}+\frac{i\bar{\zeta}}{2}e^{i\theta}.
\end{align}
Here we get $ G_{\theta}^s(\bar{\zeta})=\dfrac{i}{2\bar{\zeta}}e^{i\theta} $ and $ H_{\theta}^s(\zeta)=\dfrac{i}{2\zeta}e^{-i\theta} $ they also satisfy
$\overline{G_{\theta}^s(\bar{\zeta})}=-H_{\theta}^s(\zeta)$.
Hence
\begin{align}\label{phitheta}
\varphi_{\theta}^s=-\frac{i}{2}\ln(\zeta)e^{-i\theta}+\frac{i}{2}\ln(\bar{\zeta})e^{i\theta}.
\end{align}
Equations \eqref{xthetaminus}, \eqref{xthetaplus} and \eqref{phitheta} describes the general solution of Born-Infeld equation \eqref{borninfeldequn}. Therefore, $X_{\theta}^s:=(x_{\theta}^s,t_{\theta}^s,
\varphi_{\theta}^s)$ gives a one parameter family of Born-Infeld solitons.
 
\section{Some Identities}

Let $ X $ and $ A $ be complex, where A is not an odd multiple of $ \frac{\pi}{2} $. Then
\begin{equation}\label{ramanujancos}
\dfrac{\cos(X+A)}{\cos(A)}=\prod_{k=1}^\infty\left\lbrace\left(1-\dfrac{X}{(k-\frac{1}{2})\pi-A}\right)\left(1+\dfrac{X}{(k-\frac{1}{2})\pi+A}\right)\right\rbrace 
\end{equation}
If $ X $ and $ A $ are real, then
\begin{equation}\label{ramanujantan}
\tan^{-1}(\tanh X\cot A)=\tan^{-1}\left(\frac{X}{A}\right) + \sum_{k=1}^{\infty}\left(\tan^{-1}\left(\frac{X}{k\pi+A}\right)-\tan^{-1}\left(\frac{X}{k\pi-A}\right)\right).
\end{equation}
The above identities were obtained by Srinivasa Ramanujan \cite{ramanujan}.
We are going to use this identity to arrive at further nontrivial identities, using Weierstrass-Enneper representation for maximal surfaces.

The Weierestrass-Enneper representation for a maximal surface $ (x,y,z), $ in Lorentz-Minkowski space $\mathbb{L}^3:=(\mathbb{R}^3, dx^2+dy^2-dz^2),$ whose Gauss map is one-one is given by \cite{kobasingu},
\begin{center}
$x(\zeta)=Re(\int^\zeta M(\omega)(1+\omega^{2})d\omega)\text{ ; }y(\zeta)=Re(\int^\zeta iM(\omega)(1-\omega^{2})d\omega)$
\end{center}
\begin{center}
 $z(\zeta)=Re(\int^\zeta- 2M(\omega)\omega d\omega),$ where $ \zeta=u+iv. $
\end{center}
\subsection{Identity corresponding to Scherk's surface of first kind}
\begin{proposition}
For $ \zeta \in \Omega \subset  {\mathbb C} -\{\pm 1, \pm i \}$, we have the following identity
\begin{equation}\label{1stidentity}
\ln\vert \frac{{\zeta}^2-1}{{\zeta}^2+1}\vert=\sum_{k=1}^{\infty}\ln \left(\dfrac{(k-\frac{1}{2})\pi-i\ln\vert \frac{\zeta-i}{\zeta+i}\vert}{(k-\frac{1}{2})\pi-i\ln\vert \frac{\zeta+1}{\zeta-1}\vert}\right)+\sum_{k=1}^{\infty}\ln \left(\dfrac{(k-\frac{1}{2})\pi+i\ln\vert \frac{\zeta-i}{\zeta+i}\vert}{(k-\frac{1}{2})\pi+i\ln\vert \frac{\zeta+1}{\zeta-1}\vert}\right).
\end{equation}
\end{proposition}

\begin{proof}
For Scherk's surface of first kind \cite{kobamax}, which in non-parametric form is, defined by,
\begin{equation}
z=\ln (\cosh y)-\ln (\cosh x)~~~\text{,}~~~ (\cosh^{-2}x+\cosh^{-2}y>1)
\end{equation}

 If we take the Weierstrass data, $ M(\omega)=\dfrac{2}{1-{\omega}^4}. $  Then using the Weierstrass-Enneper representation, we can write Scherk's surface in parametric form as
 \begin{align}\label{xzetascherkfirst}
x(\zeta)=\ln\vert \frac{\zeta+1}{\zeta-1}\vert,
\end{align}
\begin{align}\label{yzetascherkfirst}
 y(\zeta)=\ln\vert \frac{\zeta-i}{\zeta+i}\vert,
\end{align}
\begin{align}\label{zzetascherkfirst}
 z(\zeta)=\ln\vert \frac{{\zeta}^2-1}{{\zeta}^2+1}\vert.
\end{align} 
This parametrization is well defined on $ \Omega \subset  \mathbb{C}-\{\pm i, \pm i\}. $
We easily compute that
\begin{center}
 $x(\zeta)=\frac{1}{2}\ln\left(\frac{(u+1)^2+v^2}{(u-1)^2+v^2}\right)$ $~~\text{;}~~$ $ y(\zeta)=\frac{1}{2}\ln\left(\frac{u^2+(v-1)^2}{u^2+(v+1)^2}\right) $
\end{center}
\begin{center}
$ z(\zeta)=\frac{1}{2}\ln\left(\frac{(u^2-v^2-1)^2+4u^2v^2}{(u^2-v^2+1)^2+4u^2v^2}\right). $
\end{center}

One can easily verify from the expressions for $ x,y, $ and $ z $ that 
     \begin{center}
     $z=\ln(\cosh y)-\ln(\cosh x).$
     \end{center}
Now if we take the logarithm on both sides of the identity \eqref{ramanujancos}, we get
\begin{equation}\label{ramanujanlogcos}
\ln\left(\dfrac{\cos(X+A)}{\cos(A)}\right)=\sum_{k=1}^{\infty}\ln\left(\dfrac{(k-\frac{1}{2})\pi-(X+A)}{(k-\frac{1}{2})\pi-A}\right)+\sum_{k=1}^{\infty}\ln\left(\dfrac{(k-\frac{1}{2})\pi+(X+A)}{(k-\frac{1}{2})\pi+A}\right).
\end{equation}
If we put $ X+A=iy $ and $ A=ix $ in \eqref{ramanujanlogcos}, where $ ix $ is not an odd multiple of $-\frac{i\pi}{2}, $ we obtain
\begin{equation}\label{ramanujanzlogcos}
z=\ln \left(\dfrac{\cosh y}{\cosh x}\right)=\ln \left(\dfrac{\cos iy}{\cos i x}\right)=\sum_{k=1}^{\infty}\ln \left(\dfrac{(k-\frac{1}{2})\pi-iy}{(k-\frac{1}{2})\pi-ix}\right)+\sum_{k=1}^{\infty}\ln \left(\dfrac{(k-\frac{1}{2})\pi+iy}{(k-\frac{1}{2})\pi+ix}\right).
\end{equation}
Now we use \eqref{xzetascherkfirst}, \eqref{yzetascherkfirst}, and \eqref{zzetascherkfirst} in \eqref{ramanujanzlogcos}, we will get our first identity \eqref{1stidentity}.
\end{proof}

\subsection{Identity corresponding to helicoid of second kind.}
\begin{proposition}
For $ \zeta \in \Omega  \subset {\mathbb C} - \{0 \},$ we have the following identity
\begin{equation}\label{2ndidentity}
\frac{Im\left(\zeta+\frac{1}{\zeta}\right)}{Im\left(\zeta-\frac{1}{\zeta}\right)}=\dfrac{1}{i}\prod_{k=1}^{\infty}\left\lbrace\left(\dfrac{(k-1)\pi+i\ln|\zeta|}{(k-\frac{1}{2})\pi+i\ln|\zeta|}\right)\left(\dfrac{k\pi-i\ln|\zeta|}{(k-\frac{1}{2})\pi-i\ln|\zeta|}\right)\right\rbrace.
\end{equation}
\end{proposition}

\begin{proof}
The helicoid of second kind is a ruled surface, which in non-parametric form, is given by \cite{kobamax}
\begin{equation}\label{helicoid2}
 z=-x\tanh y,~~\text{ }~~ (x^2\leq \cosh^2 y).
 \end{equation}

Here, we use a variant of Weierstrass-Enneper representation given by \cite{kobamax}
\begin{center}
$x(\zeta)=Re(\int^\zeta M(\omega)(1+\omega^{2})d\omega)$ $~~\text{;}~~$ $ y(\zeta)=Re(\int^\zeta 2iM(\omega)\omega d\omega)$
\end{center}
\begin{center}
$ z(\zeta)=Re(\int^\zeta M(\omega)(\omega^{2}-1) d\omega)$.
\end{center} 

and we take the Weierstrass data as, $ M(\omega)=\dfrac{i}{2\omega^2} $. Then we get a parametric representation of \eqref{helicoid2}, valid in a domain $\Omega \subset {\mathbb C} - \{ 0 \}$, as follows,
\begin{align}\label{xyzheli2}
x(\zeta)=-\frac{1}{2}Im\left(\zeta-\frac{1}{\zeta}\right),~~\text{ }~~y(\zeta)=-\ln|\zeta|,~~\text{ }z(\zeta)=-\frac{1}{2}Im\left(\zeta+\frac{1}{\zeta}\right).
\end{align}
Now we write equation \eqref{helicoid2} as $ -\dfrac{z}{ix}=\dfrac{\cos(iy+\frac{\pi}{2})}{\cos(iy)},$ next replace $ X $ by $ \frac{\pi}{2} $ and $ A $ by $ iy $, in Ramanujan identity \eqref{ramanujancos}, and then use equations \eqref{xyzheli2} to get the desired identity \eqref{2ndidentity}.
\end{proof}

\subsection{Identity corresponding to Lorentzian helicoid}
\begin{proposition}
For $ \zeta=u+iv $, such that $ \zeta \in \Omega \subset {\mathbb C} - \{0\} $, we have the following identity
\begin{equation}\label{3rdidentity}
Im(\ln(\zeta))-\tan^{-1}\left(\tanh\left(-\frac{1}{2}Re\left(\zeta+\frac{1}{\zeta}\right)\right)\cot\left(\frac{1}{2}Im\left(\zeta-\frac{1}{\zeta}\right)\right)\right) 
\end{equation}
$$=\pm\frac{\pi}{2}+
\sum_{k=1}^{\infty}\left(\tan^{-1}\left(\dfrac{Re(\zeta+\frac{1}{\zeta})}{Im(\zeta-\frac{1}{\zeta})-2k\pi}\right)+\tan^{-1}\left(\dfrac{Re(\zeta+\frac{1}{\zeta})}{Im(\zeta-\frac{1}{\zeta})+2k\pi}\right)\right),$$
where the constant  term is  $\frac{\pi}{2},$  when either $ u> 0~~\text{and}~~ v>0 $ \textit{or} $ u<0~~\text{and}~~v<0 $ \textit{or} $ u=0 $ \textit{or} $ v=0$ and the constant term is  $-\frac{\pi}{2}$ otherwise.
\end{proposition}

\begin{proof}
Consider a Lorentzian helicoid, $ z=\pm\frac{\pi}{2}+\tan^{-1}(\frac{y}{x}) $, we have Weierstrass-Enneper representation for this valid in a domain $\Omega \subset {\mathbb C} - \{0\}$, given by \cite{rahul}
\begin{align}\label{xzetatan}
x(\zeta)=\frac{1}{2}Im(\zeta-\frac{1}{\zeta}),~~\text{ }~~y(\zeta)=-\frac{1}{2}Re(\zeta+\frac{1}{\zeta}),~~\text{ } z(\zeta)=Im(\ln(\zeta))
\end{align}
In the parameter $ \zeta=u+iv, $ we have $ z=\tan^{-1}(\frac{v}{u}) $ and $ \tan^{-1}(\frac{y}{x})=-\tan^{-1}(\frac{u}{v}). $ Now we see that $$z=\frac{\pi}{2}+\tan^{-1}(\frac{y}{x}), $$ only when either $ u> 0~~\text{and}~~ v>0 $ \textit{or} $ u<0~~\text{and}~~v<0 $ \textit{or} $ u=0 $ \textit{or} $ v=0. $ For other values of $ u, v$ i.e. when either $ u<0 $ and $ v>0 $ \textit{or} $ u>0 $ and $ v<0 $, we get
$$ z=-\frac{\pi}{2}+\tan^{-1}(\frac{y}{x}).$$
Next we use equations \eqref{xzetatan}  in Ramanujan identity \eqref{ramanujantan} to obtain the identity \eqref{3rdidentity}.
\end{proof}

\end{document}